\documentclass[12pt]{article}
\usepackage{amsmath,amssymb,amsfonts}
\usepackage{epsfig}
\usepackage{tikz}
\usepackage{graphics}
\usepackage{bbm}
\usepackage{bm, bbm}
\usepackage{todonotes}
\usepackage[thmmarks]{ntheorem}
\usepackage[]{hyperref}
 \hypersetup{colorlinks=true, citecolor=black,%
                 filecolor=black,%
                 linkcolor=black,%
                 urlcolor=black,  pdftex}

\usetikzlibrary[graphs, arrows, backgrounds, intersections, positioning, fit, petri, calc, shapes, decorations.pathmorphing]
\usepgflibrary{patterns}

\definecolor{darkred}{RGB}{105,0,0}

\makeatletter
\newtheoremstyle{prime}%
 {\item[\hskip\labelsep \theorem@headerfont ##1\ \theorem@separator]}%
{\item[\hskip\labelsep \theorem@headerfont ##1\ ##3' \theorem@separator]}
\makeatother

\makeatletter
\newtheoremstyle{proofof}
{\item[\hskip\labelsep \theorem@headerfont ##1\ \theorem@separator]}%
{\item[\hskip\labelsep \theorem@headerfont ##1\ ##3\theorem@separator]}
\makeatother

\newtheorem{theorem}{Theorem}
\newtheorem{lemma}[theorem]{Lemma}

\newtheorem{proposition}[theorem]{Proposition}

\newtheorem{corollary}[theorem]{Corollary}

\theoremstyle{prime}

\def \QD1 {\hfill $\spadesuit$}

\newcommand{\DF}[1]{{\bf #1\/}}

\newcommand{\set}[2]{\{#1 \;|\; #2 \}}
\newcommand{\ems}{\varnothing}
\newcommand{\sm}{\setminus}

\newcommand{\De}{\Delta}
\newcommand{\de}{\delta}

\newcommand{\cn}{\chi}
\newcommand{\lcn}{\chi^{\ell}}

\newcommand{\cV}{{\cal V}}

\newcommand{\nato}{\mathbb{N}_0}

\numberwithin{equation}{section}

\theoremstyle {nonumberplain}
\theoremseparator{:}
\theorembodyfont{\normalfont}
\theoremsymbol{\rule{1ex}{1ex}}
\newtheorem{proof}{Proof}

\theoremstyle{proofof}
\newtheorem{proofof}{Proof of}

\theoremsymbol{$\square$}

\begin{document}
\title{\bf Generalized Hypergraph Coloring}

\author{{{
Thomas Schweser}\thanks{partially supported by DAAD, Germany (as part of BMBF) and by the Ministry of Education Science, Research and Sport of the Slovak Republic within the project 57320575}
\thanks{
Technische Universit\"at Ilmenau, Inst. of Math., PF 100565, D-98684 Ilmenau, Germany. E-mail
address: thomas.schweser@tu-ilmenau.de}}
}

\date{}
\maketitle

\begin{abstract}
A smooth hypergraph property $\mathcal{P}$ is a class of hypergraphs that is hereditary and non-trivial, i.e., closed under induced subhypergraphs and it contains a non-empty hypergraph but not all hypergraphs. In this paper we examine $\mathcal{P}$-colorings of hypergraphs with smooth hypergraph properties $\mathcal{P}$. A $\mathcal{P}$-coloring of a hypergraph $H$ with color set $C$ is a function $\varphi:V(H) \to C$ such that $H[\varphi^{-1}(c)]$ belongs to $\mathcal{P}$ for all $c \in C$. Let $L: V(H) \to 2^C$ be a so called list-assignment of the hypergraph $H$. Then, a $(\mathcal{P},L)$-coloring of $H$ is a $\mathcal{P}$-coloring $\varphi$ of $H$ such that $\varphi(v) \in L(v)$ for all $v \in V(H)$. The aim of this paper is a characterization of $(\mathcal{P},L)$-critical hypergraphs. Those are hypergraphs $H$ such $H-v$ is $(\mathcal{P},L)$-colorable for all $v \in V(H)$ but $H$ itself is not. Our main theorem is a Gallai-type result for critical hypergraphs, which implies a Brooks-type result for $(\mathcal{P},L)$-colorable hypergraphs. In the last section, we prove a Gallai bound for the degree sum of $(\mathcal{P},L)$-critical locally linear hypergraphs.
\end{abstract}



\section{Introduction and main results}
All hypergraphs considered in this paper are finite, undirected, and loopless but may contain multiple edges. Let $\mathcal{H}$ denote the class of all those hypergraphs. A \textbf{hypergraph property} $\mathcal{P}$ is a  isomorphism-closed subclass of $\mathcal{H}$; $\mathcal{P}$ is said to be \textbf{smooth} if $\mathcal{P}$ is closed under induced subhypergraphs (i.e., $\mathcal{P}$ is \textbf{hereditary}, and $\mathcal{P}$ contains a non-empty hypergraph, but not all hypergraphs (i.e., $\mathcal{P}$ is \textbf{non-trivial}. For graphs, lots of research has been done on the topic of coloring with respect to hereditary properties already (see \cite{BoBrMi97},  \cite{BoDrMi95}, \cite{MiSkre01}).

In the 1960s, Erd\H os and Hajnal \cite{ErdHaj66} introduced a coloring concept for hypergraphs. According to them, a \textbf{proper coloring} of a hypergraph $H$ is a function $\varphi: V(H) \to \mathbb{N}$ such that for each (hyper-)edge $e$ there are vertices $u,v$ contained in $e$ such that $\varphi(u) \neq \varphi(v)$. Since each edge of a graph contains exactly two vertices, this concept is a generalization of the usual coloring concept for graphs. Moreover, this definition enables the transfer of various famous results on colorings of graphs to the hypergraph case. For example, Brooks' Theorem \cite{brooks} was extended to hypergraphs by Jones \cite{Jones} in 1975.   

In this paper we regard the $\mathcal{P}$-list-coloring problem for hypergraphs. A \textbf{$\mathcal{P}$-coloring} of a hypergraph $H$ with color set $C$ is a function $\varphi:V(H) \to C$ such that for each $c \in C$ the subhypergraph $H[\varphi^{-1}(c)]$ belongs to $\mathcal{P}$. Given a list assignment $L:V(H) \to 2^C$, a \textbf{$(\mathcal{P},L)$ coloring} of $H$ is a $\mathcal{P}$-coloring $\varphi$ of $H$ such that $\varphi(v) \in L(v)$ for all $v \in V(H)$. 
The $\mathcal{P}$\textbf{-list-chromatic number} $\chi^\ell(H:\mathcal{P})$ of a hypergraph $H$ as the least integer $k$ such that $H$ is $(\mathcal{P},L)$-colorable for all list-assignments $L$ with $|L(v)| \geq k$ for all $v \in V(H)$.
It is notable that the $\mathcal{P}$-list-coloring problem is a natural extension of the ordinary list-coloring problem, where we consider the subclass $\mathcal{P}=\mathcal{O}$ of $\mathcal{H}$ consisting of all edgeless hypergraphs, and so 
$\chi^\ell(H:\mathcal{O})$ corresponds to the ordinary \textbf{list-chromatic number} $\chi^\ell(H)$ of $H$. For graphs, list-colorings were introduced by Erd\H os, Rubin, and Taylor in 1979 \cite{ErdRubTay79}.

When regarding colorings of graphs and hypergraphs, it is often useful to consider critical (hyper-)graphs. Following Dirac \cite{Dir52}, \cite{Dir53}, a graph $G$ is (vertex) $k$-critical if $\chi(G-v)<\chi(G)=k$ for every $v\in V(G)$. The hypergraph-equivalent was introduced by Lov\'asz \cite{Lov68}.

The aim of this paper is to extend various basic results for the list-chromatic number of hypergraphs. In particular, we present a Brooks-type result for the $\mathcal{P}$-list-chromatic number and a Gallai-type result for $(\mathcal{P},L)$-critical hypergraphs, i.e. hypergraphs $H$ that do not admit a $(\mathcal{P},L)$-coloring, but for each $v \in V(H)$ the subypergraph $H-v$ is $(\mathcal{P},L)$-colorable. In the last section, a bound for the number of edges in locally linear critcal hypergraphs is proven; the bound resembles Gallais bound for the class of chromatic critical graphs. 

\subsection{Notation and Basic Concepts}
In this paper, we will mainly use the notation of Schweser and Stiebitz \cite{SStieb}.
A \textbf{hypergraph} is a triple $H=(V,E,i)$, whereas $V$ and $E$ are two finite sets and $i:E \to 2^V$ is a function with $|i(e)| \geq 2$ for $e \in E$. Then, $V(H)=V$ is the \textbf{vertex set} of $H$ and its elements are the \textbf{vertices} of $H$. Furthermore, $E(H)=E$ is the \textbf{edge set} of $H$; its elements are the \textbf{edges} of $H$. Lastly, the mapping $i_H=i$ is the \textbf{incidence function} of $H$ and $i_H(e)$ is the set of vertices that are \textbf{incident} to the edge $e$ in $H$. The \textbf{empty} hypergraph is the hypergraph $H$ with $V(H)=E(H)=\ems$; we denote it by $H=\ems$.

For a hypergraph $H$ we use the following notation. The \textbf{order} $|H|$ of $H$ is the number of vertices of $H$. Let $e$ be an arbitrary edge of $H$. If $|i_H(e)| \geq 3$, the edge $e$ is said to be a \textbf{hyperedge}, otherwise, i.e. for $|i_H(e)|=2$, $e$ is an \textbf{ordinary} edge. Two edges $e,e'$ are \textbf{parallel}, if $e \neq e'$ and $i_H(e)=i_H(e')$. A \textbf{simple} hypergraph is a hypergraph without parallel edges. As usual, a $q$\textbf{-uniform} hypergraph $H$ is a hypergraph with $|i_H(e)|=q$ for all $e \in E$. Thus, a \textbf{graph} is just a $2$-uniform hypergraph; i.e. each edge is ordinary. As for hypergraphs, a \textbf{simple graph}  is a graph without parallel edges.

A hypergraph $H'$ is a \DF{subhypergraph} of $H$, written $H'\subseteq H$, if $V(H')\subseteq V(H)$, $E(H')\subseteq E(H)$, and $i_{H'}=i_H|_{E(H')}$.
Moreover, $H'$ is a \textbf{proper} subhypergraph of $H$, if $H' \subseteq H$ and $H' \neq H$ holds. Let $H_1$ and $H_2$ be two subhypergraphs of $H$. Then, $H_1 \cup H_2$ denotes the \textbf{union} of $H_1$ and $H_2$, that is, the subhypergraph of $H$ with $V(H')=V(H_1) \cup V(H_2)$, $E(H')=E(H_1)\cup E(H_2)$, and $i_{H'}=i_H|_{E(H')}$. Similarly, $H'=H_1 \cap H_2$ denotes the \textbf{intersection} of $H_1$ and $H_2$, it holds $V(H')=V(H_1) \cap V(H_2)$, $E(H')=E(H_1)\cap E(H_2)$, and $i_{H'}=i_H|_{E(H')}$.

Another important operation for the class of hypergraphs is the so called \textbf{merging}. Given two \textbf{disjoint} hypergraphs $H_1$ and $H_2$, that is, $V(H^1) \cap V(H^2) = \ems$ and $E(H^1) \cap E(H^2) = \ems$, and a vertex $v^*$ that is neither in $V(H^1)$ nor in $V(H^2)$, we define a new hypergraph $H$ as follows. Let $V(H)=((V(H^1) \cup V(H^2))\sm \{v^1,v^2\}) \cup \{v^*\}$, $E(H)=E(H^1) \cup E(H^2)$, and
$$i_H(e)=
\begin{cases}
i_{H^j}(e) & \text{if } e\in E(H^j), v^j \not \in i_{H^j}(e)~(j \in \{1,2\}),\\
(i_{H^j}(e)\sm \{v^j\}) \cup \{v^*\} & \text{if } e\in  E(H^j), v^j \in i_{H^j}(e)~(j \in \{1,2\}).
\end{cases}$$
In this case, we say that $H$ is obtained from $H^1$ and $H^2$ by merging $v^1$ and $v^2$ to $v^*$.

Let $H$ be a hypergraph and let $X \subseteq V(H)$ be a vertex set. We consider two new hypergraphs. First, $H[X]$ is the subhypergraph of $H$ with
$$V(H[X])=X, E(H[X])=\set{e\in E}{i_H(e)\subseteq X} \mbox{, and } i_{H[X]}=i_H|_{E(H[X])}.$$
 We say that $H[X]$ is the subhypergraph of $H$ \textbf{induced} by $X$.  More general, a hypergraph $H'$ is said to be an \textbf{induced subhypergraph} of $H$ if $V(H') \subseteq V(H)$ and $H'=H[V(H')]$. Secondly, $H(X)$ is the hypergraph with
$$V(H(X))=X, E(H(X))=\set{e\in E}{|i(e)\cap X|\geq 2},$$
and
$$i_{H(X)}(e)=i_H(e) \cap X \mbox{ for all } e\in E(H(X)). $$
We say that $H(X)$ is the hypergraph obtained by \textbf{shrinking} $H$ \textbf{to} $X$. Note that $H(X)$ does not necessarily need to be a subhypergraph of $H$. As usual, we define $H-X=H[V(H) \setminus X]$ and $H \div X = H(V(H) \setminus X)$. For the sake of readability, if $X=\{v\}$ for some vertex $v$, we will write $H - v$ and $H \div v$ instead of $H - X$ and $H \div X$. To obtain the reverse operation to $H-v$, let $H'$ be a proper induced subhypergraph of $H$ and let $v \in V(H) \setminus V(H')$. Then, $H'+v=H[V(H') \cup \{v\}]$.

Let $H$ be a non-empty hypergraph. A \textbf{hyperpath of length} $q$ in $H$ is a sequence $(v_1, e_1, v_2, e_2, \ldots, v_q, e_q, v_{q+1})$ of distinct vertices
$v_1, v_2, \ldots, v_{q+1}$ of $H$ and distinct edges
$e_1, e_2, \ldots, e_q$ of $H$ such that
$\{v_i,v_{i+1}\}\subseteq i_H(e_i)$ for $i=1, 2, \ldots, q$. The hypergraph $H$ is \textbf{connected} if there is a hyperpath in $H$ between any two of its vertices. A \textbf{component} of $H$ is a maximal connected subhypergraph of $H$. A \textbf{separating vertex} of $H$ is a vertex $v \in V(H)$ such that $H$ is the union of two induced subhypergraph $H_1$ and $H_2$ with $V(H_1)\cap V(H_2)=\{v\}$ and $|H_i|\geq 2$ for $i\in \{1,2\}$. Note that $v$ is a separating vertex if and only if $H \div v$ has more components than $H$. Regarding edges, an edge $e$ is a \textbf{bridge} of a hypergraph $H$, if $H-e$ has $|e|-1$ more components than $H$. Finally, a \textbf{block} of $H$ is a maximal connected subhypergraph of $H$ that has no separating vertex. Thus, every block of $H$ is a connected induced subhypergraph of $H$. It is easy to see that two blocks of $H$ have at most one vertex in common and that a vertex $v$ is a separating vertex of $H$ if and only if it is contained in more than one block. By $\mathcal{B}(H)$ we denote the set of all blocks of $H$.

As usual, we write $H=K_n$ if $H$ is a complete graph of order $n$ and $H=C_n$ if $H$ is a cycle of order $n$ consisting only of ordinary edges. A cycle $C_n$ is called \textbf{odd} or \textbf{even} depending on whether its order $n$ is odd or even. Lastly, given a simple hypergraph $H$ and an integer $t \geq 1$, we denote by $H'=tH$ the hypergraph which results from $H$ by replacing each edge of $H$ by $t$ parallel edges.

\subsection{Degeneracy of hypergraphs}
For a hypergraph $H$ and a vertex $v$ from $V(H)$, let
$$E_H(v)=\set{e\in E(H)}{v\in i_H(e)}.$$
The \textbf{degree} of $v$ in $H$ is defined as $d_H(v)=|E_H(v)|$.  As usual, $\de(H)=\min_{v\in V(H)} d_H(v)$ is the \DF{minimum degree} of $H$ and $\De(H)=\max_{v\in V(H)}d_H(v)$ is the \DF{maximum degree} of $H$. If $H$ is empty, we set $\de(H)=\De(H)=0$. Furthermore, the degree-sum over all vertices of $H$ is denoted by $$d(H) = \sum_{v \in V(H)} d_H(v).$$ A non-empty hypergraph $H$ is said to be $r$-\textbf{regular} or, briefly,   \textbf{regular} if each vertex in $H$ has degree $r$.

If $e$ is an ordinary edge of $H$ with $i_H(e)=\{u,v\}$, we brievly write $e=uv$ and $e=vu$. The \textbf{multiplicity} of two distinct vertices $u$ and $v$ in $H$ is defined by
$$\mu_H(u,v)=|\set{e\in E(H)}{e=uv}|.$$
Note that if $v\in V(H)$, then every vertex $u\in V(H)\sm \{v\}$ satisfies
\begin{align*}
d_{H\div v}(u)=d_H(u)-\mu_H(u,v).
\end{align*}

In order to prove our main result in Section~1.5, we need some results related to degeneracy. We say that a hypergraph $H$ is \textbf{strictly} $k$\textbf{-degenerate} ($k \geq 0$), if in every non-empty subhypergraph $H'$ of $H$ there is a vertex $v$ such that $d_{H'}(v) < k$. Thus, $H$ is strictly $0$-degenerate if and only if $H = \ems$, and $H$ is strictly $1$-degenerate if and only if $E(H) = \ems$. A natural extension of degeneracy can be obtained by regarding functions instead of a fixed integer. Let $H$ be a hypergraph and let $h:V(H) \to \mathbb{N}_0$. We say that $H$ is \textbf{strictly} $h$\textbf{-degenerate} if in each non-empty subhypergraph $H'$ of $H$ there is a vertex $v$ such that $d_{H'}(v) < h(v)$.

\subsection{Partitions and colorings of hypergraphs}

Let $H$ be a hypergraph and let $p \geq 1$ be an integer. A $p$\textbf{-partition} or just \textbf{partition} of $H$ is a sequence $(H_1,H_2,\ldots,H_p)$ of pairwise induced subhypergraphs of $H$ with $V(H)=V(H_1) \cup V(H_2) \cup \ldots \cup V(H_p)$; the subhypergraphs $H_i$ are called \textbf{parts} of the partition. Note that a part may be empty.

A \textbf{coloring} of $H$ with \textbf{color set} $C$ is a function $\varphi:V(H) \to C$. If $|C|=k$, we also say that $\varphi$ is a $k$\textbf{-coloring} of $H$. For $c \in C$, the set $\varphi^{-1}(c)=\{v \in V(H) ~ | ~ \varphi(v)=c\}$ is called a \textbf{color class} of $H$ with respect to $\varphi$. A first natural extension of the coloring concept is to assign each vertex a list of colors from which the color of the vertex has to be chosen. More formally, given a hypergraph $H$ and a color set $C$, a \textbf{list-assignment} $L$ is a function from $V(H)$ to $2^C$. An $L$\textbf{-coloring} of $H$ is a coloring $\varphi$ of $H$ such that $\varphi(v) \in L(v)$ for all $v \in V(H)$. Of course, a $p$-partition $(H_1,H_2,\ldots,H_p)$ of a hypergraph $H$ can always be regarded as a coloring $\varphi$ of $H$ with color set $\{1,2,\ldots,p\}$ and vice versa; the color classes $\varphi^{-1}(c)$ correspond to the parts $H_c=H[\varphi^{-1}(c)]$.

Coloring of graphs and hypergraphs is a huge topic within graph theory and various well-known restrictions have been examined already. For example, a \textbf{proper coloring} or \textbf{proper} $L$\textbf{-coloring} of a hypergraph $H$ is a coloring, respectively $L$-coloring of $H$, such that each color class induces an edgeless subhypergraph of $H$. The \textbf{chromatic number} $\chi(H)$ of a hypergraph $H$ is the least integer $k$ such that $H$ admits a proper $k$-coloring. Similarly, the \textbf{list-chromatic number} $\chi^\ell(H)$ is the least integer $k$ such that $H$ admits a proper $L$-coloring for each list assignment $L$ satisfying $|L(v)| \geq k$ for all $v \in V(H)$. Since $\lcn(H) = k$ implies that $H$ has a proper $L$-coloring for the constant list-assignment $L$ with $L(v) = \{1,2, \ldots, k\}$, it clearly holds $\cn(H)\leq \lcn(H)$. For simple graphs, the list-chromatic number was introduced independently by Vizing \cite{Vizing} and Erd\H os, Rubin and Taylor \cite{ErdRubTay79} (they use the term \textbf{choice number}).

\subsection{Hypergraph Properties}
Let $\mathcal{H}$ be the class of all hypergraphs. A \textbf{hypergraph property} $\mathcal{P}$ is a subclass of $\mathcal{H}$ that is closed under isomorphisms. In this section, we regard a special type of hypergraph properties. We say that $\mathcal{P}$ is a \textbf{smooth} hypergraph property, if the following two conditions hold.

\begin{itemize}
\item[(P1)] $\mathcal{P}$ is \textbf{hereditary}, i.e., $\mathcal{P}$ is closed under induced subhypergraphs, and
\item[(P2)] $\mathcal{P}$ is \textbf{non-trivial}, i.e., $\mathcal{P}$ contains a non-empty hypergraph but is different from $\mathcal{H}$.
\end{itemize}

Hereditary properties for graphs have been studied extensively, an interesting overview can be found in \cite{BoBrFr97}. Some important hereditary properties that are smooth, in particular, are the following:
\begin{align*}
\mathcal{O} &= \{H \in \mathcal{H} ~|~ H \text{ is edgeless} \},\\
\mathcal{S}_k &= \{H \in \mathcal{H} ~|~ \Delta(H) \leq k\}, \text{ and}\\
\mathcal{D}_k &= \{H \in \mathcal{H} ~|~ H \text{ is strictly } (k+1)\text{-degenerate} \}
\end{align*}
with $k \geq 0$.
For a smooth hypergraph property $\mathcal{P}$ let $$\mathcal{F}(\mathcal{P})=\{H ~|~ H \not \in \mathcal{P}, \text{ but } H-v \in \mathcal{P} \text{ for all } v \in V(H)\},$$ and let $$d(\mathcal{P})=\min \{\delta(H) ~|~ H \in \mathcal{F}(\mathcal{P})\}.$$

The statements of the next proposition are well-known for graphs and easy to extend to hypergraphs.

\begin{proposition}\label{prop_smoothprop}
Let $\mathcal{P}$ be a smooth hypergraph property. Then, the following statements hold:

\begin{itemize}
\item[\upshape (a)] $\mathcal{P}$ contains $K_0$ and $K_1$.
\item[\upshape (b)] A hypergraph $H$ belongs to $\mathcal{F}(\mathcal{P})$ if and only if each proper induced subhypergraph of $H$ belongs to $\mathcal{P}$, but $H$ does not.
\item[\upshape (c)] A hypergraph $H$ does not belong to $\mathcal{P}$ if and only if $H$ contains an induced subhypergraph from $\mathcal{F}(\mathcal{P})$.
\item[\upshape (d)] The class $\mathcal{F}(\mathcal{P})$ is non-empty and $d(\mathcal{P})$ is from $\mathbb{N}_0$.
\item[\upshape (e)] If a hypergraph $H$ does not belong to $\mathcal{P}$, but $H-v \in \mathcal{P}$ for some $v \in V(H)$, then $d_H(v) \geq d(\mathcal{P})$.
\end{itemize}
\end{proposition}

\begin{proof}
Since $\mathcal{P}$ is non-trivial, $\mathcal{P}$ contains a non-empty hypergraph $H$. As $\mathcal{P}$ is hereditary, it contains all induced subhypergraphs of $H$ and, therefore, $K_0$ and $K_1$. Statement (b) follows from (P1) and the definition of $\mathcal{F}(\mathcal{P})$ since $H-v$ is a proper induced subhypergraph of $H$ for all $v \in V(H)$. In order to prove (c), let $H$ be a hypergraph. If $H$ contains an induced subhypergraph $G$ from $\mathcal{F}(\mathcal{P})$, then clearly $H \not \in \mathcal{P}$ (by (P1)). Conversely, if $H$ does not belong to $\mathcal{P}$, there is an induced subhypergraph $G$ of $H$ such that $G \not \in \mathcal{P}$ and $|G|$ is minimum. Then, $G-v \in \mathcal{P}$ for all $v \in V(G)$ and $G$ belongs to $\mathcal{F}(\mathcal{P})$. Since $\mathcal{P}$ is different from $\mathcal{H}$ (by (P2)), statement (d) is an immediate consequence of (c).

It remains to prove statement (e). To this end, let $H \not \in \mathcal{P}$ be a hypergraph such that $H-v \in \mathcal{P}$ for some $v \in V(H)$. By (c), $H$ contains a subhypergraph $G$ from $\mathcal{F}(\mathcal{P})$. Then, $G$ contains $v$, since otherwise $G$ would be an induced subhypergraph of $H-v$ and would belong to $\mathcal{P}$ (by (P1)). Thus, $$d(\mathcal{P}) \leq \delta(G) \leq d_G(v) \leq d_H(v),$$ which proves (e).
\end{proof}

Hypergraph properties can be useful in order to generalize coloring concepts for hypergraphs. Let $\mathcal{P}$ be an arbitrary hypergraph property and let $C$ be a color set. We say that a coloring $\varphi: V(H) \to C$ is a $\mathcal{P}$\textbf{-coloring} of the hypergraph $H$, if each color class $\varphi^{-1}(c)$ induces a hypergraph belonging to  $\mathcal{P}$  ($c \in C$). Furthermore, the $\mathcal{P}$\textbf{-chromatic number} $\chi(H:\mathcal{P})$ of $H$ is the least integer $k$ such that $H$ admits a $\mathcal{P}$-coloring with color set $\{1,2,\ldots,k\}$. Similar, given a hypergraph $H$, a color set $C$, and a list-assignment $L: V(H) \to 2^C$, a $(\mathcal{P},L)$\textbf{-coloring} of $H$ is an $L$-coloring $\varphi$ of $H$ such that $H[\varphi^{-1}(c)] \in \mathcal{P}$ for all $c \in C$. If $H$ admits a $(\mathcal{P}, L)$-coloring, we also say that $H$ is $(\mathcal{P}, L)$\textbf{-colorable}. Finally, we define the $\mathcal{P}$\textbf{-list-chromatic number} $\chi^\ell(H:\mathcal{P})$ of a hypergraph $H$ as the least integer $k$ such that $H$ is $(\mathcal{P},L)$-colorable for all list-assignments $L$ with $|L(v)| \geq k$ for all $v \in V(H)$. Note that the case $\mathcal{P}=\mathcal{O}$ corresponds to proper ($L$-)colorings.

If $\mathcal{P}$ is a smooth hypergraph property, then $K_0,K_1 \in \mathcal{P}$, which implies that $$\chi(H:\mathcal{P}) \leq \chi^\ell(H:\mathcal{P}) \leq |H|$$ for all hypergraphs $H$. Moreover, it holds $$\chi^\ell(H:\mathcal{P}) - 1 \leq \chi^\ell(H-v: \mathcal{P}) \leq \chi^\ell(H:\mathcal{P})$$ for all hypergraphs $H$ and for each vertex $v \in V(H)$. The second inequality is obvious. In order to obtain the first inequality, assume that $\chi^\ell(H,\mathcal{P})=k$, but $\chi^\ell(H-v:\mathcal{P}) \leq k-2$ for some vertex $v \in V(H)$, that is, $H-v$ is $(\mathcal{P},L')$-colorable for each list-assignment $L'$ such that $|L'(u)| \geq k-2$ for all $u \in V(H - v)$. Now let $L$ be an arbitrary list-assignment for $H$ with $|L(u)| \geq k-1$ for all $u \in V(H)$. Then, we may assign $v$ an arbitrary color $c$ from $L(v)$ and set $L'(u)=L(u) \setminus \{c\}$ for all $u \in V(H) \setminus \{v\}$. As a consequence, $L'$ is a list-assignment for $V(H-v)$ such that $|L'(u)| \geq k-2$ for all $u \in V(H-v)$ and, thus, $H-v$ admits an $L'$-coloring, which leads  to an $L$-coloring of $H$. Since $L$ was chosen arbitrarily, this implies that $\chi^\ell(H:\mathcal{P}) \leq k-1$, a contradiction.

Let $L$ be a list-assignment for a hypergraph $H$. We say that $H$ is $(\mathcal{P},L)\textbf{-critical}$ if $H-v$ is $(\mathcal{P},L)$-colorable for all $v \in V(H)$, but $H$ itself is not.

\begin{proposition}\label{prop_pl-vertex-crit}
Let $\mathcal{P}$ be a smooth graph property with $d(\mathcal{P})=r$, let $H$ be a non-empty hypergraph, and let $L$ be a list-assignment for $H$. If $H$ is $(\mathcal{P},L)$-critical, then the following conditions hold:
\begin{itemize}
\item[\upshape (a)] $d_H(v) \geq r|L(v)|$ for all $v \in V(H)$.
\item[\upshape (b)] Let $v$ be a vertex of $H$ with $d_H(v)=r|L(v)|$, and let $\varphi$ be a $(\mathcal{P},L)$-coloring of $H-v$ with color set $C$. Moreover, for $c \in L(v)$, let $$H_{c,v}=H[\varphi^{-1}(c) \cup \{v\}] \text{  and } d_c = d_{H_{c,v}}(v)$$ Then, $d_c=r$ for all $c \in L(v)$ and $E_H(v)= \bigcup_{c \in L(v)} E_{H_{c,v}}(v).$
\end{itemize}
\end{proposition}

\begin{proof}
Let $v$ be an arbitrary vertex of $H$. Since $H$ is $(\mathcal{P},L)$-critical, there is a $(\mathcal{P},L)$-coloring $\varphi$ of $H-v$. As $H$ is not $(\mathcal{P},L)$-colorable, it holds that $H[\varphi^{-1}(c) \cup \{v\}]$ is not in $\mathcal{P}$ for all $c \in L(v)$, and thus, by Proposition~\ref{prop_smoothprop}(e), $$r=d(\mathcal{P}) \leq d_{H[\varphi^{-1}(c) \cup \{v\}]} (v) = d_c$$ for each $c \in L(v)$. Consequently, we obtain $$d_H(v) \geq \sum_{c \in L(v)} d_c \geq r|L(v)|.$$ This proves (a). If $v$ is a vertex of $H$ with $d_H(v) = r|L(v)|,$ then the above inequalities immediately imply that $r=d_c$ for all $c \in L(v)$ and that $E_H(v)= \bigcup_{c \in L(v)} E_{H_{c,v}}(v)$, which proves (b).
\end{proof}

Let $\mathcal{P}$ be a smooth hypergraph property with $d(\mathcal{P})=r$, let $H$ be a hypergraph, and let $L$ be a list-assignment for $H$ such that $H$ is $(\mathcal{P},L)$-critical. By $V(H,\mathcal{P},L)$, we denote the set of vertices $v \in V(H)$ with $d_H(v)=r|L(v)|$ in $H$. A vertex $v \in V(H)$ is said to be a \textbf{low vertex} if $v \in V(H,\mathcal{P},L)$, and a \textbf{high vertex}, otherwise. Moreover, we call $H(V(H,\mathcal{P},L))$ the \textbf{low-vertex hypergraph} with respect to $(H,\mathcal{P},L)$. Note that $H(V(H,\mathcal{P},L))$, contrary to the case for graphs, is not necessarily a subhypergraph of $H$. Our main result is a Gallai-type theorem that characterizes the structure of the low-vertex hypergraph. For simple graphs, it was obtained in 1995 by Borowiecki, Drgas-Burchardt and Mih\'ok \cite{BoDrMi95}. We say that a hypergraph $H$ is a \textbf{brick}, if $H=tC_n$ for some $t \geq 1$ and $n \geq 3$ odd or $H=tK_n$ for some $t,n \geq 1$.

\begin{theorem}\label{theorem_main-result}
Let $\mathcal{P}$ be a smooth hypergraph property with $d(\mathcal{P})=r$, let $H$ be a non-empty hypergraph, and let $L$ be a list-assignment for $H$ such that $H$ is $(\mathcal{P},L)$-critical and $F=H(V(H,\mathcal{P},L))$ is non-empty. If $B$ is a block of $F$, then $B$ is a brick, or $B \in \mathcal{F}(\mathcal{P})$ and $B$ is $r$-regular, or $B \in \mathcal{P}$ and $\Delta(B) \leq r$.
\end{theorem}

The proof of Theorem~\ref{theorem_main-result} is presented in the next section. In the remaining part of this section, we will show how to use the above theorem in order to obtain a Brooks-type result for the $\mathcal{P}$-chromatic number as well as for the $\mathcal{P}$-list-chromatic number. To this end, let $\mathcal{P}$ be a smooth hypergraph property. We say that a hypergraph $H$ is $(\chi^\ell, \mathcal{P})$\textbf{-critical} if $\chi^\ell(G:\mathcal{P}) < \chi^\ell(H:\mathcal{P})$ for each proper induced subhypergraph $G$ of $H$. Note that $H$ is $(\chi^\ell, \mathcal{P})$-critical if and only if $\chi^\ell(H-v:\mathcal{P})=\chi^\ell(H:\mathcal{P})-1$ for each vertex $v \in V(H)$.

\begin{lemma}\label{lemma_brooks}
If $\mathcal{P}$ is a smooth hypergraph property with $d(\mathcal{P})=r \geq 1$, then the following statements hold:
\begin{itemize}
\item[\upshape (a)] For each hypergraph $H$ there is a $(\chi^\ell,\mathcal{P})$-critical induced subhypergraph $G$ such that $\chi^\ell(G:\mathcal{P})=\chi^\ell(H:\mathcal{P}).$
\item[\upshape (b)] If $H$ is a $(\chi^\ell, \mathcal{P})$-critical hypergraph with $\chi^\ell(H:\mathcal{P})=k$, then $\delta(H) \geq r(k-1)$. Moreover, if $U= \{v \in V(H) ~|~ d_H(v) = r(k-1) \}$ is non-empty, then each block $B$ of $H(U)$ is a brick, or $B \in \mathcal{F}(\mathcal{P})$ and $B$ is $r$-regular, or $B \in \mathcal{P}$ and $\Delta(B) \leq r$.
\item[\upshape (c)] For each hypergraph $H$ it holds $\chi^\ell(H:\mathcal{P}) \leq \frac{\Delta(H)}{r} + 1$.
\end{itemize}
\end{lemma}

\begin{proof}
We can choose an induced subhypergraph $G$ of $H$ with $\chi^\ell(G:\mathcal{P}) = \chi^\ell(H:\mathcal{P})$ whose order is minimum; this hypergraph clearly fulfills statement (a). To prove (b), let $H$ be a $(\chi^\ell, \mathcal{P})$-critical hypergraph with $\chi^\ell(H:\mathcal{P})=k$ and let $U=\{v \in V(H) ~ | ~ d_H(v) = r(k-1)\}.$ Then, there exists a list-assignment $L$ of $H$ with $|L(v)|=k-1$ for all $v \in V(H)$ such that $H$ is not $(\mathcal{P},L)$-colorable, but $H-v$ is $(\mathcal{P},L)$-colorable for each $v \in V(H)$. As a consequence, $H$ is $(\mathcal{P},L)$-critical and, by Proposition~\ref{prop_pl-vertex-crit}(a), it holds $\delta(H) \geq r(k-1)$ and $U=V(H,\mathcal{P},L).$
Applying Theorem~\ref{theorem_main-result} then leads to each block $B$ of $G(U)$ having the structure that is required in (b).

For the proof of (c), let $H$ be an arbitrary hypergraph with $\chi^\ell(H:\mathcal{P})=k$. By (a), $H$ contains a $(\chi^\ell,\mathcal{P})$-critical induced subhypergraph $G$ such that $\chi^\ell(G:\mathcal{P})=\chi^\ell(H:\mathcal{P}).$ By (b), $G$ has minimum degree at least $r(k-1)$ and we conclude $\Delta(H) \geq \Delta(G) \geq \delta(G) \geq r(k-1)$ and, hence, $\chi^\ell(H:\mathcal{P}) \leq \frac{\Delta(H)}{r} + 1$.
\end{proof}

We say that a hypergraph property $\mathcal{P}$ is \textbf{additive} if $\mathcal{P}$ is closed under vertex disjoint unions. This means that a non-empty hypergraph $H$ is in $\mathcal{P}$ if and only if each component of $H$ is in $\mathcal{P}$. If we also require $\mathcal{P}$ to be smooth, then each hypergraph $H$ from $\mathcal{F}(\mathcal{P})$ is connected and it holds $d(\mathcal{P}) \geq 1$ (since $K_0,K_1 \in \mathcal{P}$ by Proposition~\ref{prop_smoothprop}(a)).

Recall that $\mathcal{O}$ is the class of edgeless hypergraphs. The property $\mathcal{O}$ obviously is non-trivial, hereditary and additive, and $\mathcal{O} \subseteq \mathcal{P}$ holds for each property $\mathcal{P}$ that is smooth and additive (by Proposition~\ref{prop_smoothprop}(a)). As a consequence, each hypergraph $H$ satisfies $\chi^\ell(H:\mathcal{P}) \leq \chi^\ell(H:\mathcal{O}) = \chi^\ell(H)$ for any smooth and additive hypergraph property $\mathcal{P}$. With the help of Lemma~\ref{lemma_brooks} we are able to give a Brooks-type result for smooth and additive hypergraph properties. This theorem was proven for simple graphs in \cite{BoDrMi95}.

\begin{theorem}\label{theorem_brooks}
Let $\mathcal{P}$ be a non-trivial, hereditary and additive hypergraph property with $d(\mathcal{P})=r$ and let $H$ be a connected hypergraph. Then, $$\chi^\ell(H:\mathcal{P}) \leq \left \lceil \frac{\Delta(H)}{r} \right \rceil + 1,$$
and if equality holds, then $H=tK_{(kr+t)/t}$ for some integers $t \geq 1, k \geq 0$, or $H$ is a $tC_n$ for $t=r, n \geq 3$ odd and $\chi^\ell(H:\mathcal{P})=3$, or $H$ is $r$-regular and $H \in \mathcal{F}(\mathcal{P})$.
\end{theorem}

\begin{proof}
Let $H$ be an arbitrary connected hypergraph. If $\Delta(H)$ is not divisible by $r$, then the statement follows directly from Lemma~\ref{lemma_brooks}(c) (in particular, equality cannot hold). Thus, we may assume $\Delta(H) = kr$ for some integer $k \geq 0$ and so $\chi^\ell(H:\mathcal{P}) \leq k + 1$ (by Lemma~\ref{lemma_brooks}(c)). If $\chi^\ell(H:\mathcal{P}) \leq k$, there is nothing left to show. Suppose $\chi^\ell(H:\mathcal{P}) = k + 1$. Then, by Lemma~\ref{lemma_brooks}(a),(b), $H$ contains a $(\chi^\ell,\mathcal{P})$-critical subhypergraph $G$ satisfying $\chi^\ell(G:\mathcal{P})=k + 1$ and $\delta(G) \geq kr$. As $H$ is connected and as $\Delta(G) \leq \Delta(H)=kr$, this implies that $H=G$ and, hence, $H$ is $kr$-regular and $(\chi^\ell,\mathcal{P})$-critical. Thus, $H=H(U)$, whereas $U=\{v \in V(H) ~ | ~ d_H(v)=rk \}$ and, by Lemma~\ref{lemma_brooks}(b), each block $B$ of $H$ is a brick, or $B \in \mathcal{F}(\mathcal{P})$ and $B$ is $r$-regular, or $B \in \mathcal{P}$ and $\Delta(B) \leq r$. As $H$ itself is $kr$-regular, this clearly implies that $H$ is a block.

If $H=tK_n$ with $t,n \geq 1$, then $d_H(v)=t(n-1)=kr$ and thus $n=\frac{kr+t}{t}$. Hence, we are done. If $H=tC_n$ for some $t\geq 1$ and $n \geq 3$ odd, we have $kr=2t \geq 2$. In the case $k=1$, it follows $\chi^\ell(H:\mathcal{P})=2$ and $r=2t$. As $H$ is $(\chi^\ell, \mathcal{P})$-critical, this implies that $H$ is in $\mathcal{F}(\mathcal{P})$ and $H$ is $r$-regular. For $k\geq 2$, we argue as follows. Since $\chi^\ell(H:\mathcal{P}) \leq \chi^\ell(H) \leq 3$ and as $\chi^\ell(H:\mathcal{P})=k + 1$, it must hold $\chi^\ell(H:\mathcal{P})=3$, $k=2$ and, thus, $r=t$. Hence, we are done.

If $H \in \mathcal{F}(\mathcal{P})$ and $H$ is $r$-regular, then $k=1$ (as $H$ is $kr$-regular), and we are done, too. Finally, if $H \in \mathcal{P}$ and $\Delta(H) \leq r$, then $\chi^\ell(H:\mathcal{P})=1$, but $k=1$, contradicting the premise. This completes the proof.
\end{proof}

In the previously mentioned paper by Erd\H os, Rubin and Taylor \cite{ErdRubTay79}, a degree version of Brooks' Theorem is proven. To conclude this section, we present a related result to theirs.

\begin{theorem}
Let $\mathcal{P}$ be a non-trivial, hereditary and additive hypergraph property with $d(\mathcal{P})=r$, and let $H$ be a connected hypergraph. Moreover, let $L$ be a list-assignment for $H$ such that $r|L(v)| \geq d_H(v)$ for all $v \in V(H)$. Then, $H$ is $(\mathcal{P},L)$-colorable, unless each block $B$ of $H$ is a brick, or $B \in \mathcal{F}(\mathcal{P})$ is $r$-regular, or $B \in \mathcal{P}$ and $\Delta(B) \leq r$.
\end{theorem}

\begin{proof}
If $H$ is $(\mathcal{P},L)$-colorable, there is nothing left to show. Suppose that $H$ is not $(\mathcal{P},L)$-colorable. Then, there is a $(\mathcal{P},L)$-critical subhypergraph $G$ of $H$. By Proposition~\ref{prop_pl-vertex-crit}(a), it holds $d_G(v) \geq r|L(v)|$ for all $v \in V(G)$ and, thus, $d_G(v)=d_H(v)=r|L(v)|$ for all $v \in V(G)$. As $H$ is connected, this implies that $G=H$, i.e. $H$ is $(\mathcal{P},L)$-critical. Moreover, it follows that $d_H(v)=r|L(v)|$ for all $v \in V(H)$ and so $V(H)=V(H,\mathcal{P},L)$. Applying Theorem~\ref{theorem_main-result} completes the proof.
\end{proof}

\subsection{Proof of Theorem~\ref{theorem_main-result}}
In order to prove Theorem~\ref{theorem_main-result} we need to consider hypergraph partitions with specific constraints on the degeneracy. Let $H$ be an arbitrary hypergraph. A function $f:V(H) \to \nato^p$ is called a \DF{vector function} of $H$. By $f_i$ we name the $i$th coordinate of $f$, i.e., $f=(f_1,f_2, \ldots, f_p)$. The set of all vector functions of $H$ with $p$ coordinates is denoted by $\cV_p(H)$. For $f\in \cV_p(H)$, an \DF{$f$-partition} of $H$ is a $p$-partiton $(H_1,H_2, \ldots, H_p)$ of $H$ such that $H_i$ is strictly $f_i$-degenerate for all $i\in \{1,2, \ldots, p\}$. If the hypergraph $H$ admits an $f$-partition, then $H$ is said to be \DF{$f$-partitionable}. Schweser and Stiebitz \cite{SStieb} examined, under which conditions a hypergraph $H$ is $f$-partitionable. They used the following definitions.

Let $H$ be a connected hypergraph and let $f \in \mathcal{V}_p(H)$ be a vector-function for some $p \geq 1$. We say that $H$ is $f$\textbf{-hard}, or, equivalently, that $(H,f)$ is a \textbf{hard pair}, if one of the following conditions hold.

\begin{itemize}
\item[(1)] $H$ is a block and there exists an index $j \in \{1,2,\ldots,p\}$ such that
$$f_i(v)=
\begin{cases}
d_H(v) & \text{if i=j,}\\
0 & \text{otherwise}
\end{cases}$$
for all $i \in \{1,2,\ldots,p\}$ and for each $v \in V(H)$. In this case, we say that $H$ is a \textbf{monoblock} or a block of type \textbf{(M)}.
\item[(2)] $H=tK_n$ for some $t \geq 1, n \geq 3$ and there are integers $n_1,n_2,\ldots,n_p \geq 0$ with at least two $n_i$ different from zero such that $n_1 + n_2 + \ldots + n_p=n-1$ and that
$$f(v)=(tn_1,tn_2,\ldots,tn_p)$$
for all $v \in V(H)$. In this case, we say that $H$ is a block of type \textbf{(K)}.
\item[(3)] $H=tC_n$ with $t \geq 1$ and $n \geq 5$ odd and there are two indices $k \neq \ell$ from the set $\{1,2,\ldots,p\}$ such that
$$f_i(v)=
\begin{cases}
t & \text{if } i \in \{k,\ell\}, \\
0 & \text{otherwise}
\end{cases}
$$ for all $i \in \{1,2,\ldots,p\}$ and for each $v \in V(H)$. In this case, we say that $H$ is a block of type \textbf{(C)}.
\item[(4)] There are two hard pairs $(H^1,f^1)$ and $(H^2,f^2)$ with $f^1 \in \mathcal{V}_p(H^1)$ and $f^2 \in \mathcal{V}_p(H^2)$ such that $H$ is obtained from $H^1$ and $H^2$ by merging two vertices $v^1 \in V(H_1)$ and $v^ 2 \in V(H_2)$ to a new vertex $v^*$. Furthermore, it holds
$$f(v)=
\begin{cases}
f^1(v) & \text{if } v \in V(H_1) \sm \{v^1\}, \\
f^2(v) & \text{if } v \in V(H_2) \sm \{v^2\}, \\
f^1(v^1) + f^2(v^2) & \text{if } v=v^*
\end{cases}$$
for all $v \in V(H)$.
\end{itemize}

The next theorem was proven by Schweser and Stiebitz \cite{SStieb} in 2018, it characterizes $f$-partitionable hypergraphs $H$ fulfilling the condition $f_1(v) + f_2(v) + \ldots + f_p(v) \geq d_H(v)$ for all $v \in V(H)$.

\begin{theorem}
\label{Theorem:Hauptsatz}
Let $H$ be a connected hypergraph and let $f\in \cV_p(H)$ be a vector function with $p\geq 1$ such that $f_1(v)+f_2(v)+\cdots +f_p(v)\geq d_H(v)$ for all $v\in V(H)$. Then $H$ is not $f$-partitionable if and only if $(H,f)$ is a hard pair.
\end{theorem}

Note that if $(H,f)$ is of type (C) or (K), then, in particular, $H$ is a brick. We will use the above theorem in order to prove our main result.

\begin{proofof}[Theorem~\ref{theorem_main-result}]
Let $B$ be an arbitrary block of $F=H(V(H,\mathcal{P},L))$. Since $H$ is $(\mathcal{P},L)$-critical, there is a $(\mathcal{P},L)$-coloring $\varphi$ of $H - V(B)$ with a set $C$ of $p$ colors. By renaming the colors we may assume $C=\{1,2,\ldots,p\}$. Let $H_i=H[\varphi^{-1}(i)]$ for each $i \in \{1,2,\ldots,p\}$. Then, for $v \in V(B)$, we define the vector function $f: V(B) \to \mathbb{N}_0^p$ as follows. For each $v \in V(B)$, let $f_i(v)= \max\{0, r-d_{H_i+v}(v)\}$ if $i \in L(v)$, and $f_i(v)=0$, otherwise.

We claim that $B$ is not $f$-partitionable. Assume, to the contrary, that $B$ admits an $f$-partition $(H_1',H_2',\ldots,H_p')$. Then, for $i \in \{1,2,\ldots,p\}$ let $\tilde{H_i}=H[V(H_i) \cup V(H_i')]$. Obviously, $(\tilde{H_1}, \tilde{H_2},\ldots, \tilde{H_p})$ is a partition of $H$. Note that $v \in V(\tilde{H_i})$ implies that $i \in L(v)$ (since $f_i(v) \geq 1$ for $v \in V(H_i')$). If $\tilde{H_i}\in \mathcal{P}$ for all $i \in \{1,2,\ldots,p\}$, it follows that $H$ is $(\mathcal{P},L)$-colorable, a contradiction. As a consequence, there is an $i \in \{1,2,\ldots,p\}$ such that $\tilde{H_i} \not \in \mathcal{P}$. By Proposition~\ref{prop_smoothprop}(c), there exists an induced subhypergraph $G$ of $\tilde{H_i}$ such that  $G \in \mathcal{F}(\mathcal{P})$ and, thus, $\delta(G) \geq d(\mathcal{P})=r$. Since $H_i$ is in $\mathcal{P}$ but $G$ is not, $G$ contains a vertex of $H_i'$. Thus, the hypergraph $G'=H_i'[V(G) \cap V(H_i')]$ is non-empty. However, since $H_i'$ is strictly $f_i$-degenerate, there is a vertex $v$ in $G'$ such that $d_{G'}(v) < f_i(v) = r - d_{H_i + v} (v)$ and thus $d_G(v) \leq d_{G'}(v) + d_{H_i+v}(v) < r$, a contradiction. Hence, $B$ is not $f$-partitionable.

Since $d_H(v)=r|L(v)|$ for all $v \in V(B)$, we obtain that
\begin{align*}
\sum_{i=1}^p f_i(v) & = \sum_{i \in L(v)}f_i(v)
 \geq  \sum_{i \in L(v)}(r-d_{H_i+v}(v))\\
& =  d_H(v) - \sum_{i \in L(v)} d_{H_i+v}(v) \geq d_B(v)
\end{align*}
for all $v \in V(B)$. Thus, by Theorem~\ref{Theorem:Hauptsatz} and as $B$ is a block, $(B,f)$ is of type (M), (K) or (C). If $(B,f)$ is not of type (M), then $B$ is a brick and we are done. Thus assume that $(B,f)$ is of type (M). Then, there is exactly one index $i$ such that $f_i(v) = d_B(v)$ for all $v \in V(B)$ and $f_j(v) = 0$ for $j \neq i$ from the set $\{1,2,\ldots,p\}$. As a consequence, $d_{H_j+v}(v) \geq r$ for all $j \in L(v) \setminus\{i\}$ and thus, $d_B(v) \leq r$ for all $v \in V(B)$. If $B \in \mathcal{P}$, we have $\Delta(B) \leq r$ and there is nothing left to show. If $B \not \in \mathcal{P}$, then by Proposition~\ref{prop_smoothprop}(c), $B$ contains an induced subhypergraph $B'$ from $\mathcal{F}(\mathcal{P})$. Since $d_B(v) \leq r$ for all $v \in V(B)$ and since $\delta(B') \geq d(\mathcal{P})= r$, it must hold $B=B'$ and $d_B(v) = r$ for all $v \in V(B)$. Consequently, $B \in \mathcal{F}(\mathcal{P})$ and $B$ is $r$-regular. This completes the proof.
\end{proofof}

\section{A Gallai-type bound for the degree sum of critical linear hypergraphs}
The topic of finding lower bounds for the number of edges, respectively the degree sum of critical graphs and hypergraphs with respect to some coloring concept has already been examined extensively in the past. Regarding proper colorings of simple graphs (not hypergraphs), Gallai~\cite{Gal63b} proved that for a $(k+1)$-critical graph $G\neq K_{k+1}$, that is, a graph which has chromatic number $k+1$ but each proper subgraph has chromatic number at most $k$, it holds
$$d(G) \geq k |V(G)| + \frac{k- 2}{k^2 + 2k -2}|V(G)|$$ if $k\geq 3$. For simple hypergraphs, an even stronger bound was proven by Kostochka and Stiebitz \cite{KosStieb03}. Mih\'ok and \v{S}krekovski \cite{MiSkre01} proved a Gallai-type bound for the case of $(\mathcal{P},L)$-critical graphs. In the next section, with the help of Stiebitz and Kostochka's approach, we show that the bound also holds for $(\mathcal{P},L)$-critical locally linear hypergraphs.

Let $\mathcal{P}$ be a smooth additive hypergraph property and let $H$ be a $(\mathcal{P},L)$-critical hypergraph, whereas $L$ is a list-assignment for $H$ with $|L(v)|=k$ for all $v \in V(H)$. Then we say that $H$ is \textbf{locally linear with respect to} $(\mathcal{P},L)$ if $H(V(H,\mathcal{P},L))$ is simple. Furthermore, if $H$ is a $(\chi^\ell,\mathcal{P})$-critical hypergraph with $\chi^\ell(H:\mathcal{P})=k+1$, we say that $H$ is \textbf{locally linear with respect to} $(\chi^\ell,\mathcal{P})$ if $H$ is locally linear with respect to $(\mathcal{P},L)$ for some list-assignment $L$ with $|L(v)|=k$ for all $v \in V(H)$ such that $H$ is $(\mathcal{P},L)$-critical. Note that if $H$ is locally linear with respect to $(\mathcal{P},L)$, then $H$ is locally linear for each list-assignment $L'$ satisfying that $H$ is $(\mathcal{P},L')$-critical and that $|L'(v)|=|L(v)|$ for all $v \in V(H)$, since for the low vertex hypergraphs it clearly holds $V(H,\mathcal{P},L)=V(H,\mathcal{P},L')$. Note that if $H$ is a simple hypergraph, then the shrinking operation may still lead to parallel edges. Since it will be necessary that the low vertex hypergraph is simple, we need to limit ourselves to locally linear hypergraphs. Moreover, it is important to note that if $\mathcal{P}=\mathcal{O}$, then any $(\mathcal{P},L)$-critical hypergraph is locally linear with respect to $(\mathcal{P},L)$ (see \cite{KosStieb03}).

In the following, let $\mathcal{P}$ be a smooth additive hypergraph property with $d(\mathcal{P})=r \geq 1$, let $k \geq 1$ and let $\delta=kr$. Furthermore, let $H$ be a locally linear hypergraph with respect to  $(\chi^\ell,\mathcal{P})$ where $\chi^\ell=k+1$ for some $k \geq 1$. Let $n=|H|$ and let
$$a(\delta,n)=\delta n + \frac{\delta-2}{\delta^2 + 2 \delta - 2}n.$$
Our aim is to prove that $d(H) \geq a(\delta,n)$. Note that the $(\chi^\ell, \mathcal{P})$-critical locally linear hypergraphs for $\chi^\ell(H:\mathcal{P})=2$ (i.e. $k=1$) are exactly the hypergraphs from $\mathcal{F}(\mathcal{P})$ (by Proposition~\ref{prop_smoothprop}(b) and since $H$ being $(\chi^\ell, \mathcal{P})$-critical implies that $H$ is ($\mathcal{P},L)$-critical with $L(v)=\{1\}$ for all $v \in V(H)$). In this case, however, the boundary is not true for many properties. As an example consider the class $D_{r-1}$ of strictly $r$-degenerate hypergraphs. Then it is easy to check that $\mathcal{F}(\mathcal{P})$ contains all $r$-regular connected hypergraphs, and thus, the bound clearly does not hold for $r \geq 3$.

Thus, in the following we will assume $k \geq 2$ and, therefore, $\delta \geq 2$. If $\delta=2$, this implies $r=1$ and $k=2$. Then, $\chi^\ell(H:\mathcal{P})=3$ and, in particular, there is a list assignment $L$ for $H$ with $|L(v)|=2$ for all $v \in V(H)$ such that $H-v$ is $(\mathcal{P},L)$-colorable for all $v \in V(H)$, but $H$ is not. Consequently, $H$ is $(\mathcal{P},L)$-critical, and, by Proposition~\ref{prop_pl-vertex-crit}(a), it holds $d_H(v) \geq r|L(v)|=2$ for all $v \in V(H)$. Thus, as $\delta = 2$, it trivially holds $d(H)\geq 2n = a(2,n)$. Hence, as of now we may assume $\delta \geq 3$.
Lastly, it is important to note that if $H=K_{\delta + 1}$, then clearly $d(H) < a(\delta, n)$ for $\delta \geq 3$ and thus the bound is not true in this case. Therefore, we need to exclude the $K_{\delta + 1}$ from our further considerations.

Instead of proving the bound for $(\chi^\ell, \mathcal{P})$-critical hypergraphs, we prove a slightly stronger result regarding $(\mathcal{P},L)$-critical hypergraphs.

\begin{theorem}\label{theorem_gallai-bound}
Let $\mathcal{P}$ be a smooth additive hypergraph property with $d(\mathcal{P})=r \geq 1$, let $k \geq 2$, and let $\delta=kr \geq 3$. Furthermore, let $H \neq K_{\delta + 1}$ be a locally linear hypergraph with respect to $(\mathcal{P},L)$, whereas $L$ is a list-assignment for $H$ with $|L(v)|=k$ for all $v \in V(H)$. Then, it holds $d(H) \geq a(\delta,|H|)$.
\end{theorem}

The remaining part of this section is dedicated to the proof of the above theorem. For $(\chi^\ell, \mathcal{P})$-critical hypergraphs, we can directly conclude the next corollary from Theorem~\ref{theorem_gallai-bound}.

\begin{corollary}
Let $\mathcal{P}$ be a smooth additive hypergraph property with $d(\mathcal{P})=r \geq 1$, let $k \geq 2$, and let $\delta=kr \geq 3$. Furthermore, let $H \neq K_{\delta + 1}$ be a locally linear hypergraph with respect to $(\chi^\ell,\mathcal{P})$, whereas $\chi^\ell(H)=k+1$. Then, it holds $d(H) \geq a(\delta,|H|)$.
\end{corollary}

The proof of Theorem~\ref{theorem_gallai-bound} is mainly done via three lemmas. At first, we show that the bound always holds if a specific condition is fulfilled. Afterwards, we prove that this condition is always true. Most parts of the next three lemmas are similar to those in the paper of Kostochka and Stiebitz~\cite{KosStieb03}. To start with, we need some new notation. Since we only regard linear hypergraphs, the structures described in Theorem~\ref{theorem_main-result} can be simplified. Therefore, we say that a connected simple hypergraph $H$ is a \textbf{Gallai tree}, if each block $B$ of $H$ is a complete graph, or $B$ is a cycle of odd length, or $B \in \mathcal{F}(\mathcal{P})$ and $B$ is $r$-regular, or $B \in \mathcal{P}$ and $\Delta(B) \leq r$.

\begin{lemma} \label{lemma_dgeqa}
Let $\mathcal{P}$ be a smooth additive hypergraph property with $d(P)=r \geq 1$, let $k \geq 2$, and let $\delta=kr \geq 3$. Furthermore, let $H \neq K_{\delta + 1}$ be a locally linear hypergraph with respect to $(\mathcal{P},L)$, whereas $L$ is a list-assignment for $H$ with $|L(v)|=k$ for all $v \in V(H)$. Moreover, let $$U=\{v \in V(H)~|~d_H(v)=\delta \},$$ let $$r_\delta=\delta - 1 + \frac{2}{\delta},$$ and let $$\sigma=|U| r_\delta - d(H(U)).$$
 If $\sigma \geq 0$, then it holds $$d(H) \geq a(\delta,n).$$
\end{lemma}

\begin{proof} By Proposition~\ref{prop_pl-vertex-crit}(a), we have $\delta(H) \geq \delta$ and, thus, $U=V(H,\mathcal{P},L)$. Moreover, we claim $U \neq V(H)$. Otherwise, $H=H(U)$ would be a $\delta$-regular Gallai tree (by Theorem~\ref{theorem_main-result} and since $H$ is connected), and this is only possible if $H=K_{\delta + 1}$ (as $\delta > r$, $\delta \geq 3$). Hence, $U \neq V(H)$.

If $U = \varnothing$, we obtain $d(H) \geq (\delta + 1) n \geq a(\delta,n)$ and there is nothing left to prove. Thus, we may assume $U \neq \varnothing$.
%
Then, it holds
\begin{align*}
d(H) & = \delta |U| + \sum_{v \in V(H) \setminus U}d_H(v) \\
& \geq d(H-U) + 2\delta |U| - d(H(U))\\
&  = d(H-U) + \sigma + (2\delta - r_\delta)|U|\\
& = d(H-U) + \sigma + (\delta + 1 - \frac{2}{\delta})|U|\\
& \geq (\delta + 1 - \frac{2}{\delta})|U|
\end{align*}
On the other hand, $$d(H) \geq (\delta + 1)n  - |U|.$$
As a consequence, we obtain
\begin{align*}
d(H) + d(H)(\delta + 1 - \frac{2}{\delta}) & \geq (\delta + 1 - \frac{2}{\delta})|U| + (\delta + 1)(\delta + 1 - \frac{2}{\delta})n\\
& - |U|(\delta + 1 - \frac{
2}{\delta})\\
& = (\delta + 1)(\delta + 1 - \frac{2}{\delta})n
\end{align*}
By rearranging the inequation we easily get the required result.
\end{proof}

Thus, the only remaining question is if $\sigma \geq 0$ is always fulfilled. That this is indeed the case, is proven in the next two lemmas.

First of all, let $r_\delta=\delta - 1 + \frac{2}{\delta}$. Moreover, for an arbitrary hypergraph $F$, let $$\sigma(F)=|V(F)|r_\delta - d(F).$$ Regarding a locally linear hypergraph $H$ with respect to $(\mathcal{P},L)$, we know that each component of $H(V(H,P,L))$ forms a Gallai tree (by Theorem~\ref{theorem_main-result}). Thus, let $\mathfrak{T}_\delta$ denote the set of Gallai trees distinct from $K_{\delta+1}$ with maximum degree at most $\delta$. Lastly, for $T \in \mathfrak{T}_\delta$ and for an end-block $B$ of $T$, we define $T_B=T-(V(B)- \{x\})$, whereas $x$ denotes the only separating vertex of $T$ in $B$ (if $T$ has only one block choose an arbitrary vertex $x$ of $V(T)$).

\begin{lemma}\label{lemma:block-sigma}
Let $T \in \mathfrak{T}_\delta$ and let $\delta \geq 3$. Then, the following statements hold:
\begin{itemize}
\item[\upshape (a)] If $B \in \mathcal{B}(T)$, then $\sigma(B)=2$ if $B=K_\delta$ and $\sigma(B) \geq r_\delta$ otherwise.
\item[\upshape (b)] If $B$ is an end-block of $T$, then $\sigma(T) = \sigma(T_B) + \sigma(B) - r_\delta$.
\end{itemize}
\end{lemma}

\begin{proof}
If $B$ is a $K_b$ for some $b \in \{1,2,\ldots,\delta\}$, then $$\sigma(B)=b(r_\delta - b + 1)
\begin{cases}
\geq r_\delta, & \text{if } 1 \leq b \leq \delta-1, and\\
= 2, & \text{if } b=\delta.
\end{cases}$$
Otherwise, if $B$ is a cycle of odd length with at least $5$ vertices, then it is easy to check that $$\sigma(B) = |V(B)|(r_\delta - 2) \geq 5(r_\delta - 2) \geq r_\delta.$$  If $B=(e,\{e\})$ for some edge $e$, then $\sigma(B)=|e|(r_\delta - 1) \geq r_\delta$ (as $r_\delta \geq 2$).

It remains to consider the case that $B$ is a block with $\Delta(B) \leq r$ that is not of the above mentioned types. This implies, in particular, that $|V(B)| \geq 3$. If $k \geq 3$, then $rk \geq 2r + 1$ and we conclude
\begin{align*}
\sigma(B) & = |V(B)|(rk - 1 + \frac{2}{rk}) - \sum_{v \in V(B)}d_B(v)\\
& \geq |V(B)|(rk-1 + \frac{2}{rk}) - |V(B)| r\\
& = |V(B)|(r(k-1) - 1 + \frac{2}{rk}) \\
& \geq 2rk - 2r - 2 + \frac{4}{rk} \\
& = r_\delta + rk - 2r - 1 + \frac{2}{rk} \geq r_\delta.
\end{align*}
Otherwise, $k=2$ and, since $\delta \geq 3$, we have $r \geq 2$. Then, since $|V(B)| \geq 3$, we get
\begin{align*}
\sigma(B) & \geq |V(B)|(r(k-1) - 1 + \frac{2}{rk} )\\
& \geq 3rk - 3r - 3 + \frac{6}{rk} \\
& = r_\delta + 2rk - 3r - 2 + \frac{4}{rk} \geq r_\delta,
\end{align*}
as $2rk = 4r \geq 3r + 2$.
Due to the fact that $T_B$ and $B$ share exactly one vertex, statement (b) is evident.
\end{proof}

Following Gallai, we say that a hypergraph is an $\varepsilon_\delta$-hypergraph if each separating vertex belongs to exactly two blocks, one being a $K_\delta$ and the other one being of the form $(e,\{e\})$ for some edge $e$, and if each non-separating vertex is contained in a block, which is a $K_\delta$.

\begin{lemma}\label{lemma_sigma}
Let $T \in \mathfrak{T}_\delta$ and let $\delta \geq 4$. Then, $\sigma(T) \geq 2$ if $T$ is an $\varepsilon_\delta$-hypergraph and $\sigma(T) \geq r_\delta$, otherwise.
\end{lemma}

\begin{proof}
The proof is by induction on the number $m$ of blocks of $T$. If $m=1$, the statement follows immediately from Lemma~\ref{lemma:block-sigma}. Assume $m \geq 2$. If $T$ is an $\varepsilon_\delta$-hypergraph, then $T_B$ is not an $\varepsilon_k$-hypergraph for any end-block $B$ of $T$ and, by Lemma~\ref{lemma_brooks} we have $\sigma(T) \geq \sigma(T_B) + \sigma(B) - r_\delta \geq 2$ (as $\sigma(T_B) \geq r_\delta$ by the induction hypothesis).

If $T$ is not an $\varepsilon_\delta$-hypergraph, assume that $T$ has a block $B$ of the form $B=(e,\{e\})$. Then, clearly $e$ is a bridge of $T$. For $x \in e$, let $T_x$ denote the component of $T - \{e\}$ containing $x$. As $T$ is not an $\varepsilon_\delta$-hypergraph, $T_x$ is not an $\varepsilon_\delta$-hypergraph for at least one $x \in e$. Moreover, $r_\delta \geq \delta - 2 \geq 2$. By applying the induction hypothesis, we conclude
$$\sigma(T) = \sum_{x \in e} \sigma(T_x) - |e| \geq 2(|e|-1) + r_\delta - |e| \geq r_\delta.$$
If $T$ has no block of the form $(e,\{e\})$, then no block of $T$ is a $K_\delta$. Let $B$ be an end-block of $T$. Then, $T_B$ is not a $\varepsilon_\delta$-hypergraph and, by the induction hypothesis and Lemma~\ref{lemma:block-sigma}, $\sigma(T) = \sigma(T_B) + \sigma(B) - r_k \geq r_k$.
\end{proof}

Now we can finally prove Theorem~\ref{theorem_gallai-bound}.

\begin{proofof}[Theorem~\ref{theorem_gallai-bound}]
Let $\mathcal{P},r,k,\delta$ be defined as
in Theorem~\ref{theorem_gallai-bound} and let $H\neq K_{\delta + 1}$ be a locally linear hypergraph with respect to $(\mathcal{P},L)$, whereas $L$ is a list-assignment for $H$ satisfying $|L(v)|=k$ for all $v \in V(H)$. By Proposition~\ref{prop_pl-vertex-crit}, $H$ has minimum degree at least $\delta$. As before, let $U=\{v \in V(H) ~|~ d_H(v)=\delta\}$. Then, each component of $H(U)$ is a Gallai tree (by Theorem~\ref{theorem_main-result}) and, since $H \neq K_{\delta + 1}$, each component of $H(U)$ belongs to $\mathcal{T}_\delta$. Thus, for each component $C$ of $H(U)$ it holds $\sigma(C) \geq 2$ by Lemma~\ref{lemma_sigma}. As a consequence, $\sigma(H(U)) \geq 0$ and, by Lemma~\ref{lemma_dgeqa}, we conclude $d(H) \geq a(\delta,|V(H)|)$.
\end{proofof}


\begin{thebibliography}{99}

\bibitem{BoBrFr97}
M.~Borowiecki, I.~Broere, M.~Frick, P.~Mih\'ok and G. Semani\v{s}in,
A survery of hereditary properties of graphs,
\emph{Discuss. Mathematicae Graph Theory} \textbf{17} (1997) 5--50.

\bibitem{BoBrMi97}
M.~Borowiecki, I.~Broere, and P.~Mih\'ok,
On Generalized list colourings of graphs,
\emph{Discuss. Mathematicae Graph Theory} \textbf{17} (1995) 127--132.

\bibitem{BoDrMi95}
M.~Borowiecki, E.~Drgas-Burchardt and P.~Mih\'ok,
Generalized list colouring of graphs, \emph{Discuss. Mathematicae Graph Theory} \textbf{15} (1995) 185--193.

\bibitem{brooks}
R.~L.~Brooks,
On colouring the nodes of a network, \emph{Proc. Cambridge Philos. Soc., Math. Phys. Sci.} 37 (1941) 194--197.

\bibitem{Dir52}
G.~A.~Dirac, 
A property of 4-chromatic graphs and some remarks on critical graphs,
\emph{J. London Math. Soc.} \textbf{27} (1952) 85--92.

\bibitem{Dir53}
G.~A.~Dirac,
The structure of $k$-chromatic graphs,
\emph{Fund. Math.} \textbf{40} (1953) 42--55.


\bibitem{ErdHaj66}
P.~Erd\H os and A.~Hajnal, On the chromatic number of graphs and set-systems,
\emph{Acta Math. Acad. Sci. Hungar.} \textbf{17} (1966) 61--99.

\bibitem{ErdRubTay79}
P.~Erd\H os, A.L.~Rubin, and H.~Taylor, Choosability in graphs,
\emph{Congr. Numer.} \textbf{XXVI} (1979) 125--157.


\bibitem{Gal63b} T.~Gallai, Kritische Graphen II.
{\em Publ. Math. Inst. Hungar. Acad. Sci.}  \textbf{8} (1963) 373--395.


\bibitem{Jones}
R.~P.~Jones, Brooks' Theorem for hypergraphs,
\emph{Proc. 5th British Combinatorial Conf.} (1975) 379--384.


\bibitem{KosStieb03}
A.V.~Kostochka and M.~Stiebitz, A new lower bound on the number of edges in colour-critical graphs and hypergraphs,
\emph{J. Combin. Theory Ser. B} \textbf{87} (2003) 374--402.


\bibitem{KoStiWi96}
A.~V.~Kostochka, M.~Stiebitz and B.~Wirth, The colour theorems of Brooks and
Gallai extended, {\em Discrete Math.} {\bf 191} (1996), 125--137.

\bibitem{Lov68}
L.~Lov\'asz,
\emph{On chromatic number of finite set-systems,}
\emph{Acta Math. Acad. Sci. Hungar.} \textbf{19} (1968) 59--67.

\bibitem{MiSkre01}
P.~Mih\'ok and R.~\v{S}krekovski, Gallai's inequality for critical graphs of reducible hereditary properties,
\emph{Discuss. Mathematicae Graph Theory} \textbf{21} (2001) 167--177.


\bibitem{SStieb}
T.~Schweser and M.~Stiebitz, Hypergraph partitions and variable degeneracy, \emph{arXiv preprint arXiv:1804.04894} (2017).


\bibitem{Vizing}
V.~G.~Vizing,
Vertex coloring with given colors (auf Russisch),
\emph{Diskret. Analiz.} \textbf{29} (1976) 3--10.


\end{thebibliography}
\end{document}